\documentclass[11pt]{amsart} \usepackage{latexsym, amssymb, stmaryrd}
\usepackage[T1]{fontenc}
\usepackage{amssymb,amsmath,latexsym,enumerate, dsfont}
\usepackage{graphicx,verbatim,enumerate,%
ifpdf,mdwlist}
\usepackage{color}
\ifpdf
\usepackage{hyperref}
\else
\usepackage[hypertex]{hyperref}
\fi
\usepackage{graphicx}
\usepackage{cancel}
\usepackage{float}
\renewcommand{\phi}{\varphi}

\newtheorem{thm}{Theorem}[section]
\newtheorem{lemma}[thm]{Lemma}
\newtheorem{prop}[thm]{Proposition}
\newtheorem{cor}[thm]{Corollary}
\newtheorem{question}[thm]{Question}
\newtheorem{fact}[thm]{Fact}
\newtheorem{facts}[thm]{Facts}

\theoremstyle{definition}
\newtheorem{df}[thm]{Definition}
\newtheorem*{thm*}{Theorem}

\newenvironment{notation}
{\begin{trivlist} \item[] {\bf Notation}. } {\end{trivlist}}

\makeatletter

\def\dotminussym#1#2{%
  \setbox0=\hbox{$\m@th#1-$}%
  \kern.5\wd0%
  \hbox to 0pt{\hss\hbox{$\m@th#1-$}\hss}%
  \raise.6\ht0\hbox to 0pt{\hss$\m@th#1.$\hss}%
  \kern.5\wd0}
\newcommand{\dotminus}{\mathbin{\mathpalette\dotminussym{}}}

\DeclareMathOperator{\tp}{tp}

\newcommand{\N}{\mathbb{N}}
\newcommand{\R}{\mathbb{R}}

\def \l{\mathcal{L}}
\def \Th{\operatorname{Th}}
\def \m{\mathcal{M}}
\def \u{\mathcal{U}}

\def \n{\mathcal{N}}

\def \V{\tilde{V}}
\def \G{\frak{G}}
\def \depth{\operatorname{depth}}
\def \C{\mathcal{C}}
\def \tC{\tilde{\mathcal{C}}}
\def \tr{\operatorname{tr}}
\def \tp{\operatorname{tp}}
\def \v{\mathcal{V}}
\def \balpha{\boldsymbol\alpha}
\def \bbeta{\boldsymbol\beta}
\def \F{\mathcal{F}}
\def \FF{\mathcal{F}_{\operatorname{fo}}}
\def \FA{\mathcal{F}_\forall}
\def \r{\mathcal{R}}

\title{On the theories of McDuff's II$_1$ factors}

\author[I.~Goldbring]{Isaac Goldbring}
\address{Department of Mathematics, Statistics, and Computer Science, University of Illinois at Chicago, Science and Engineering Offices M/C 249, 851 S. Morgan St.\\
Chicago, IL, 60607}
\email{\href{mailto:isaac@math.uic.edu}{isaac@math.uic.edu}}
\urladdr{\url{homepages.math.uic.edu/~isaac/}}

\author[B.~Hart]{Bradd Hart}
\address{Department of Mathematics and Statistics, McMaster University, 1280 Main St., Hamilton ON, Canada L8S 4K1}
\email{\href{mailto:hartb@mcmaser.ca}{hartb@mcmaster.ca}}
\urladdr{\url{http://ms.mcmaster.ca/~bradd/}}

\thanks{I. Goldbring was partially supported by NSF CAREER grant DMS-1349399.}
\thanks{We thank Adrian Ioana and Thomas Sinclair for useful conversations regarding this project.}

\begin{document}


\begin{abstract}
Recently, Boutonnet, Chifan, and Ioana proved that McDuff's family of continuum many pairwise nonisomorphic separable II$_1$ factors are in fact pairwise non-elementarily equivalent by proving that any ultrapowers of two distinct members of the family are nonsiomorphic.  We use Ehrenfeucht-Fraisse games to provide an upper bound on the quantifier-depth of sentences which distinguish these theories.
\end{abstract}
\maketitle

\section{Introduction}

Constructing non-isomorphic separable II$_1$ factors has an interesting history.  Murray and von Neumann \cite{MvN} gave the first example of two non-isomorphic separable II$_1$ factors by proving that the hyperfinite II$_1$ factor $\mathcal{R}$ was not isomorphic to $L(\mathbb{F}_2)$, the group von Neumannn algebra associated to the free group on two generators.  The way they proved this was by considering an isomorphism invariant, namely property Gamma, and proving that $\mathcal{R}$ has property Gamma whilst $L(\mathbb{F}_2)$ does not.  Dixmier and Lance \cite{DL} produced a new isomorphism class by constructing a separable II$_1$ factor that does have property Gamma but does not have another property, nowadays called being McDuff, that $\mathcal{R}$ does have.  Work of Zeller-Meier \cite{ZM} and Sakai \cite{S} led to several more isomorphism classes.  The lingering question remained:  are there infinitely many isomorphism classes of separable II$_1$ factors?  In \cite{MD1}, McDuff constructed a countably infinite set of isomorphism classes of separable II$_1$ factors; in the sequel \cite{MD2}, she extends her technique to construct a family $(\m_{\balpha})_{\balpha \in 2^{\omega}}$ of pairwise non-isomorphic separable II$_1$ factors.  Throughout this paper, we will refer to this family as the family of \emph{McDuff examples}.  We will describe in detail the construction of the McDuff examples later in this paper.

The model-theoretic study of tracial von Neumann algebras began in earnest in \cite{FHS}, where it was shown that both property Gamma and being McDuff are axiomatizable properties (in the appropriate continuous first-order language for studying tracial von Neumann algebras).  It follows that $\mathcal{R}$, $L(\mathbb{F}_2)$, and the Dixmier-Lance example are pairwise non-elementarily equivalent.  However, it proved difficult to find new elementary equivalence classes of II$_1$ factors, although it was generally agreed upon by researchers in the model theory of operator algebras that there should be continuum many pairwise non-elementarily equivalent II$_1$ factors.  The current authors recognized that one of the properties considered by Zeller-Meier in \cite{ZM} was axiomatizable, thus providing a fourth elementary equivalence class;  we include a proof of this observation in the last section.

In the recent paper \cite{BCI}, Boutonnet, Chifan, and Ioana  prove that the McDuff examples are pairwise non-elementarily equivalent.  They do not, however, exhibit sentences that distinguish these examples.  Indeed, their main result is the following:  if $\balpha,\bbeta \in 2^\omega$ are distinct, then for any nonprincipal ultrafilters $\u$, $\v$ on arbitrary index sets, one has that $\m_{\balpha}^\u\not\cong \m_{\bbeta}^\v$.  It is now routine to see that $\m_{\balpha}$ and $\m_{\bbeta}$ are not elementarily equivalent.  Indeed, since the question of whether or not $\m_{\balpha}$ and $\m_{\bbeta}$ are elementarily equivalent is absolute, one can safely assume CH, whence $\m_{\balpha}$ elementarily equivalent to $\m_{\bbeta}$ would imply that, for any nonprincipal ultrafilter $\u$ on $\N$, one has that $\m_{\balpha}^\u$ and $\m_{\bbeta}^\u$ are saturated models of the same theory and a familiar back-and-forth argument shows that they are isomorphic.\footnote{For those uncomfortable with the use of CH here, one can alternatively quote the Keisler-Shelah theorem as done in \cite{BCI}.}

To a model-theorist, it is interesting to know what sentences separate these examples.  Indeed, to show that $\m_{\balpha}$ and $\m_{\bbeta}$ are not elementarily equivalent, it would be interesting to write down an explicit set of sentences $T$ such that, for some $\sigma\in T$, we have $\sigma^{\m_{\balpha}}\not=\sigma^{\m_{\bbeta}}$.  At the end of this paper, we show how to do this when $\balpha(0)\not=\bbeta(0)$; for the general case, we do not know how to do this.

The main result of this paper is instead quantitative in nature.  For II$_1$ factors $\m$ and $\n$ and $k\geq 1$, we say that $\m\equiv_k \n$ if $\sigma^\m=\sigma^\n$ for any sentence $\sigma$ of ``complexity'' at most $k$.  (The precise notion of complexity will be defined in the next section.)  Here is our main result:

\begin{thm*}
Suppose that $\balpha,\bbeta\in 2^{\omega}$ are distinct and $k\in \omega$ is least such that $\balpha(k)\not=\bbeta(k)$.  Then $M_{\balpha}\not\equiv_{5k+3}M_{\bbeta}$. 
\end{thm*}

In the next section, we describe the needed facts from logic as well as the parts of the paper \cite{BCI} that we will use in our argument.  In Section 3, we prove the main result; the proof uses Ehrenfeucht-Fr\"aisse games.  In Section 4, we take care of some miscellaneous facts.  First, we write down an explicit list of sentences that distinguish $M_{\balpha}$ from $M_{\bbeta}$ when $\balpha(0)\not=\bbeta(0)$.  Next we discuss how the model-theoretic behavior of ``good unitaries'' underlies much of the argument in \cite{BCI}.  We then go on to show how Zeller-Meier's notion of \emph{inner asymptotic commutativity} is axiomatizable and discuss another of Zeller-Meier's notions (which he does not name but we call ``super McDuff''), giving some evidence as to why it might be axiomatizable.  
Finally, we bring up the notion of the \emph{first-order fundamental group} of a II$_1$ factor and show how finding a II$_1$ factor with proper first-order fundamental group would give a different proof of the existence of continuum many theories of II$_1$ factors.

We list here some conventions used throughout the paper.  First, we follow set theoretic notation and view $k\in \omega$ as the set of natural numbers less than $k$:  $k=\{0,1\ldots,k-1\}$.  In particular, $2^k$ denotes the set of functions $\{0,1,\ldots,k-1\}\to \{0,1\}$.  If $\balpha\in 2^k$, then we set $\alpha_i:=\balpha(i)$ for $i=0,1,\ldots,k-1$ and we let $\balpha^\#\in 2^{k-1}$ be such that $\balpha$ is the concatenation of $(\alpha_0)$ and $\balpha^\#$.  If $\balpha\in 2^\omega$, then $\balpha|k$ denotes the restriction of $\balpha$ to $\{0,1,\ldots,k-1\}$.  

Whenever we write a tuple $\vec x$, it will be understood that the length of the tuple is countable (that is, finite or countably infinite).

We use $\subset$ (as opposed to $\subseteq$) to denote proper inclusion of sets.

If $\m$ is a von Neumann algebra and $A$ is a subalgebra of $\m$, then $$A'\cap \m:=\{x\in \m \ | \ [x,a]=0 \text{ for all } a\in A\}.$$  In particular, the center of $\m$ is $Z(M):=\m'\cap \m$.  For a tuple $\vec a$ from $\m$, we write $C(\vec a)$ to denote $A'\cap \m$, where $A$ is the subalgebra of $\m$ generated by the coordinates of $\vec a$.  (Technically, this notation should also mention $\m$, but the ambient algebra will always be clear from context, whence we omit any mention of it in the notation.)    
\section{Preliminaries}

\subsection{Logic}

\begin{df}
We define the \emph{quantifier-depth} $\depth(\varphi)$ of a formula $\varphi$ by induction on the complexity of $\varphi$.
\begin{itemize}
\item If $\varphi$ is atomic, then $\depth(\varphi)=0$.
\item If $\varphi_1,\ldots,\varphi_n$ are formulae, $f:\R^n\to \R$ is a continuous function and $\varphi=f(\varphi_1,\ldots,\varphi_n)$, then $\depth(\varphi)=\max_{1\leq i\leq n}\depth(\varphi_i)$.
\item If $\varphi=\sup_{\vec x} \psi$ or $\varphi=\inf_{\vec x} \psi$, then $\depth(\varphi)=\depth(\psi)+1$.
\end{itemize}
\end{df}

The main tool in this paper is the following variant of the usual Ehrenfeucht-Fraisse game.
\begin{df}
Let $\m$ and $\n$ be $\l$-structures and let $k\in \N$.  $\G(\m,\n,k)$ denotes the following game played by two players.  First, player I plays either a tuple $\vec{x_1}\in \m$ or a tuple $\vec{y_1}\in \n$.  Player II then responds with a tuple $\vec{y_1}\in \n$ or $\vec{x_1}\in \m$.  The play continues in this way for $k$ rounds.  We say that \emph{Player II wins $\G(\m,\n,k)$} if there is an isomorphism between the substructures generated by $\{\vec{x_1},\ldots,\vec{x_k}\}$ and $\{\vec{y_1},\ldots,\vec{y_k}\}$ that maps $\vec{x_i}$ to $\vec{y_i}$.  
\end{df}

\begin{df}
Suppose that $\m$ and $\n$ are $\l$-structures.
\begin{enumerate}
\item We write $\m\equiv_k \n$ if $\sigma^\m=\sigma^\n$ whenever $\operatorname{depth}(\sigma)\leq k$.
\item We write $\m\equiv_k^{EF} \n$ if II has a winning strategy for $\G(\m,\n,k)$. 
\end{enumerate}
\end{df}

It is a routine induction to show that $\m\equiv_k^{EF}\n$ implies $\m\equiv_k \n$.  

\begin{lemma}
Suppose that $\m$ and $\n$ are $\aleph_1$-saturated.  Then $\m\equiv_k \n$ if and only if $\m\equiv_k^{EF} \n$.
\end{lemma}

\begin{proof}
We prove the lemma by induction on $k$.  Suppose first that $k=0$ and that $\m\equiv_0 \n$.  Let $\m_0$ and $\n_0$ be the substructures of $\m$ and $\n$ respectively generated by the emptyset.  It follows immediately that there is an isomorphism between $\m_0$ and $\n_0$ that sends $c^\m$ to $c^\n$ for each constant symbol $c$, whence II always wins $\G(\m,\n,0)$.

Now suppose that $k>0$ and inductively assume that the lemma holds for all integers smaller than $k$.  We now describe a winning strategy for II in $\G(\m,\n,k)$.  Suppose that I first plays $\vec{a_1}\in \m$ (the case that I's first move is in $\n$ is analogous).  Consider the set $\Gamma(\vec x)$ given by
$$\Gamma(\vec x):=\{|\varphi(\vec x)-r|=0 \ : \ \depth(\varphi)<k, \ \varphi^\m(\vec{a_1})=0\}.$$  We claim that $\Gamma(\vec x)$ is finitely satisfiable in $\n$.  Towards this end, consider conditions ``$|\varphi_i(\vec x)-r_i|=0$'' in $\Gamma$, $i=1,\ldots,p$.  Let $\sigma:=\inf_{\vec x}\max_{1\leq i \leq p}|\varphi_i(\vec x)-r_i|$.  Note that $\sigma^\m=0$ and that $\depth(\sigma)\leq k$.  Since $\m\equiv_k \n$, we have that $\sigma^\n=0$, witnessing that $\Gamma$ is finitely satisfiable.  Since $\n$ is $\aleph_1$-saturated, it follows that there is $\vec{b_1}\in \n$ satisfying $\Gamma$.  The strategy for II in $\G(\m,\n,k)$ starts by demanding that II play $\vec{b_1}$.  Note now that $(\m,\vec{a_1})\equiv_{k-1}(\n,\vec{b_1})$, so by induction we have that $(\m,\vec{a_1})\equiv_{k-1}^{EF}(\n,\vec{b_1})$.  The rest of the strategy for II in $\G(\m,\n,k)$ is to have II play according the winning strategy for II in $\G((\m,\vec{a_1}),(\n,\vec{b_1}),k-1)$, where, for $p\geq 2$, round $p$ in $\G(\m,\n,k)$ is viewed as round $p-1$ in $\G((\m,\vec{a_1}),(\n,\vec{b_1}),k-1)$.  This strategy is clearly a winning strategy for II in $\G(\m,\n,k)$, whence $\m\equiv_k^{EF}\n$.   
\end{proof}

In the sequel, we will often assume that $\m\equiv_k^{EF}\n$ and that $\m$ is nonseparable.  For reasons that will become clear in the next section, we actually want to know that $\n$ is also nonseparable.  

\begin{lemma}
Suppose that $\m\equiv_2^{EF} \n$ and $\m$ is nonseparable.  Then $\n$ is nonseparable.
\end{lemma}

\begin{proof}
Let $\vec{b}$ be any tuple from $\n$.  Let I play $\vec{b}$ and have II respond with $\vec{a}$ from $\m$.  Since $\m$ is nonseparable, there is $\epsilon>0$ and $c\in \m$ such that $d(c,a_i)\geq \epsilon$ for all $i$.  Have I play $c$ and II responds with $d\in \n$.  Since II wins $\G(\m,\n,2)$, we have that $d(d,b_i)\geq \epsilon$ for all $i$, whence $\vec{b}$ is not dense in $\n$.
\end{proof}

\subsection{McDuff's examples and property $\V$}

First, we recall McDuff's examples.  Let $\Gamma$ be a countable group.  For $i\geq 1$, let $\Gamma_i$ denote an isomorphic copy of $\Gamma$ and let $\Lambda_i$ denote an isomorphic copy of $\mathbb{Z}$.  Let $\tilde{\Gamma}:=\bigoplus_{i\geq 1}\Gamma_i$.  If $S_\infty$ denotes the group of permutations of $\N$ with finite support, then there is a natural action of $S_\infty$ on $\bigoplus_{i\geq 1} \Gamma$ (given by permutation of indices), whence we may consider the semidirect product $\tilde{\Gamma}\rtimes S_\infty$.  Given these conventions, we can now define two new groups:

$$T_0(\Gamma):=\langle \tilde{\Gamma}, (\Lambda_i)_{i\geq 1} \ | \ [\Gamma_i,\Lambda_j]= \text{ for }i\geq j\rangle$$ and

$$T_1(\Gamma):=\langle \tilde{\Gamma}\rtimes S_\infty, (\Lambda_i)_{i\geq 1} \ | \ [\Gamma_i,\Lambda_j]= \text{ for }i\geq j\rangle.$$

Note that if $\Delta$ is a subgroup of $\Gamma$ and $\alpha\in \{0,1\}$, then $T_\alpha(\Delta)$ is a subgroup of $T_\alpha(\Gamma)$.  Given a sequence $\balpha\in 2^{\leq \omega}$, we define a group $K_{\balpha}(\Gamma)$ as follows:
\begin{enumerate}
\item $K_{\balpha}(\Gamma):=\Gamma$ if $\balpha=\emptyset$;
\item $K_{\balpha}(\Gamma):=(T_{\alpha_0}\circ T_{\alpha_1}\circ \cdots T_{\alpha_{n-1}})(\Gamma)$ if $\balpha\in 2^n$;
\item $K_{\balpha}$ is the inductive limit of $(K_{\balpha|n})_n$ if $\balpha\in 2^\omega$.
\end{enumerate}

We then set $\m_{\balpha}(\Gamma):=L(T_{\balpha}(\Gamma))$.  When $\Gamma=\mathbb{F}_2$, we simply write $\m_{\balpha}$ instead of $\m_{\balpha}(\mathbb{F}_2)$; these are the McDuff examples referred to the introduction.

Given $n\geq 1$, we let $\tilde{\Gamma}_{\balpha,n}$ denote the subgroup of $T_{\alpha_0}(K_{\balpha^\#}(\Gamma))$ given by the direct sum of the copies of $K_{\balpha^\#}(\Gamma)$ indexed by those $i\geq n$ and we let $P_{\balpha,n}:=L(\tilde{\Gamma}_{\balpha,n})$.  We define a \emph{generalized McDuff ultraproduct corresponding to $\balpha$ and $\Gamma$} to be an ultraproduct of the form $\prod_\u \m_{\balpha}(\Gamma)^{\otimes t_s}$ and we refer to subalgebras of the form $\prod_\u P_{\balpha,n_s}^{\otimes t_s}$ as \emph{special}.

We will need the following key facts:

\begin{facts}\label{key.fact}
Suppose that $\balpha\in 2^{<\omega}$ is nonempty, $\Gamma$ is a countable group, and $(t_s)$ is a sequence of natural numbers.
\begin{enumerate}
\item Suppose that $(m_s)$ and $(n_s)$ are two sequences of natural numbers such that $n_s<m_s$ for all $s$.  Then $(\prod_\u P_{\balpha,m_s}^{\otimes t_s})'\cap (\prod_\u P_{\balpha,n_s}^{\otimes t_s})$ is a generalized McDuff ultraproduct corresponding to $\balpha^\#$ and $\Gamma$.
\item For any sequence $(n_s)$, there is a pair of unitaries $\vec a$ from $\prod_\u \m_{\balpha}^{\otimes t_s}$ such that $\prod_\u P_{\balpha,n_s}^{\otimes t_s}=C(\vec a)$.
\item Given any separable subalgebra $A$ of $\prod_\u \m_{\balpha}^{\otimes t_s}$, there is a sequence $(n_s)$ such that $\prod_\u P_{\balpha,n_s}^{\otimes t_s}\subset A'\cap \prod_\u \m_{\balpha}^{\otimes t_s}$.
\end{enumerate}
\end{facts}

The proofs of the above facts are contained in \cite[Sections 2 and 3]{BCI}.  In particular, the proof of (1) is embedded in the proof of \cite[Lemma 3.11]{BCI}.

We recall the definition of property $\V$.

\begin{df}
Let $\m$ be a nonseparable von Neumann algebra.  We say that $\m$ has property $\V$ if there is a separable subalgebra $A\subseteq \m$ such that, for all separable $B\subseteq A'\cap \m$ and all separable $C\subseteq \m$, there is $u\in U(\m)$ such that $uBu^*\subseteq C'\cap \m$.
\end{df}

The following is \cite[Lemma 4.4]{BCI}.

\begin{fact}\label{hasV}
If $\Gamma$ is any countable group, then $\prod_\u L(T_1(\Gamma))^{\otimes t_s}$ has $\V$.
\end{fact}

\begin{notation}
If $\vec a$ and $\vec b$ are tuples from $\m$, we set $\vec a\leq \vec b$ if and only if $C(\vec b)\subseteq C(\vec a)$.  As with any preorder, we write $\vec{a}<\vec{b}$ to indicate that $\vec{a}\leq \vec{b}$ but $\vec{b}\not\leq \vec{a}$.
\end{notation}
\begin{df}
Let $k$ be a natural number.  We define what it means for a nonseparable von Neumann algebra $\m$ to have $\V$ at depth $k$:
\begin{itemize}
\item $\m$ has $\V$ at depth $0$ if it has $\V$;
\item If $k>0$, then $\m$ has $\V$ at depth $k$ if for any $\vec a$, there is $\vec b>\vec a$ such that, for all $\vec c>\vec b$, there is $\vec d>\vec c$ for which there is a von Neumann algebra $\n$ with $C(\vec c)'\cap C(\vec b)\subseteq \n \subseteq C(\vec d)'\cap C(\vec a)$ and such that $\n$ has $\V$ at depth $k-1$.  
\end{itemize}
\end{df}

In connection with this definition, let us set up some further notation.

\begin{notation}
Let $\m$ be a nonseparable von Neumann algebra and let $\vec a$, $\vec b$, $\vec c$, and $\vec d$ range over tuples from $\m$.  Furthermore, let $k\geq 1$.
\begin{enumerate}
\item For $\vec a<\vec b<\vec c< \vec d$, $\Phi(\vec a,\vec b, \vec c, \vec d; k)$ denotes the statement ``there is a von Neumann algebra $\n$ with $C(\vec c)'\cap C(\vec b)\subseteq \n \subseteq C(\vec d)'\cap C(\vec a)$ and such that $\n$ has $\V$ at depth $k-1$.'' 
\item For $\vec a<\vec b<\vec c$, $\Phi(\vec a,\vec b, \vec c;k)$ denotes the statement ``there is $\vec d>\vec c$ such that $\Phi(\vec a,\vec b,\vec c, \vec d;k)$ holds.''
\item For $\vec a<\vec b$, $\Phi(\vec a,\vec b;k)$ denotes the statement ``for all $\vec c>\vec b$, $\Phi(\vec a,\vec b,\vec c;k)$ holds.''
\item $\Phi(\vec a;k)$ denotes the statement ``there is $\vec b>\vec a$ such that $\Phi(\vec a,\vec b;k)$ holds.'' 
\end{enumerate}
\end{notation}

The definition of $\m$ having $\V$ at depth $k$ can thus be recast as:  for every $\vec a$, $\Phi(\vec a;k)$ holds.  The following is the main result of \cite{BCI} and appears there as Theorem 4.2 (really, Remark 4.3).

\begin{fact}\label{mainfact}
Suppose that $\balpha\in 2^\omega$.  Then $\m_{\balpha}^\u$ has $\V$ at depth $k$ if and only if $\alpha_k=1$.
\end{fact}

\section{The main result}

\begin{prop}\label{V3}
Suppose that $\m$ and $\n$ are nonseparable with $\m\equiv_3^{EF}\n$ and $\m$ has $\V$.  Then $\n$ has $\V$.
\end{prop}

\begin{proof}
Let $A\subseteq \m$ witness that $\m$ has $\V$.  Let $\vec a$ enumerate a countable dense subset of $A$ and let I play $\vec a$.  II then plays $\vec{a_1}\in \n$.  Let $A_1$ denote the subalgebra of $\n$ generated by $\vec{a_1}$.  We claim that $A_1$ witnesses that $\n$ has $\V$.  Towards this end, take separable $B_1\subseteq A_1'\cap \n$ and $C_1\subseteq \n$.  Let $\vec{b_1}$ and $\vec{c_1}$ enumerate countable dense subsets of $B_1$ and $C_1$ respectively.  I then plays $\vec{b_1}$ and $\vec{c_1}$. II then responds with $\vec b,\vec c\in \m$.  Since II wins, it follows that $B\subseteq A'\cap \m$, so there is $u\in U(M)$ such that $uBu^*\subseteq C'\cap \m$.  I finally plays $u$ and II responds with $u_1\in \n$.  It remains to observe that $u_1\in U(\n)$ and $u_1B_1u_1^*\subseteq C_1'\cap \n$.
\end{proof}


\begin{thm}\label{main}
Suppose that $\balpha\in 2^{k+1}$ with $\balpha(k)=1$.  Further suppose that $\Gamma$ is any countable group and that $\m$ is a generalized McDuff ultraproduct corresponding to $\balpha$ and $\Gamma$.  Finally suppose that $\m\equiv_{5k+3}^{EF} \n$.  Then $\n$ has property $\V$ at depth $k$.
\end{thm}

\begin{proof}
We proceed by induction on $k$.  Fact \ref{hasV} and the previous proposition establishes the case $k=0$.  So suppose that $k>0$ and the result holds for all smaller $k$.  
Choose a  $\vec b_0$ from $\n$ and we would like to show that $\Phi(\vec b_0;k)$ holds.  We obtain this by having player I play cooperatively in $\G(\m,\n,5k+3)$. View $\vec{b_0}$ as the first play for player I;   II responds with $\vec {a_0}$ from $\m$ according to her winning strategy. Let $P$ be a special subalgebra of $\m$ such that $P\subset C(\vec{a_0})$.  At the next round, I plays $\vec {a_1}$ from $\m$ such that $P=C(\vec{a_1})$ (so $\vec{a_0}<\vec{a_1}$) and II responds with $\vec{b_1}$ from $\n$.


\noindent \textbf{Claim 1}  $\vec{b_0}<\vec{b_1}$.  

\noindent \textbf{Proof of Claim 1:}  We show that otherwise, player I could win the game. First suppose that there is $y\in C(\vec{b_1})\setminus C(\vec{b_0})$; since $5k+3\geq 3$, we can have I play $y$ and  II responds with $x \in \m$ according to her winning strategy. We have that $x\in C(\vec{a_1})\setminus C(\vec {a_0})$, a contradiction.  This shows that $\vec{b_0}\leq \vec{b_1}$.  Now suppose that $z\in C(\vec{a_0})\setminus C(\vec{a_1})$ and have I play $z$, II responding with $w\in \n$; since II wins, it follows that $w\in C(\vec{b_0})\setminus C(\vec{b_1})$, whence $\vec{b_0}<\vec{b_1}$.

%

Now we would like to show that $\Phi(\vec{b_0},\vec{b_1};k)$ holds. Choose any $\vec{b_2}>\vec{b_1}$ and we will show $\Phi(\vec{b_0},\vec{b_1},\vec{b_2};k)$ holds.  Agreeably I plays $\vec{b_2}$, II responding with $\vec{a_2}$ from $\m$.  Since $5k+3\geq 4$, the proof of Claim 1 shows that $\vec{a_2}>\vec{a_1}$.  Now choose a special subalgebra $Q$ such that $Q \subset C(\vec{a_2})$ and $Q = C(\vec{a_3})$.  Player I now plays $\vec{a_3}$ and II responds with $\vec{b_3} \in \n$.  Since $5k+3 \geq 5$, repeating Claim 1 shows that $\vec{b_3} > \vec{b_2}$.  To finish, we show that $\Phi(\vec{b_0},\vec{b_1},\vec{b_2},\vec{b_3};k)$ holds.



Set $\m_1:=C(\vec{a_3})'\cap C(\vec{a_1})$ and $\n_1:=C(\vec{b_3})'\cap C(\vec{b_1})$.  Note that $\m_1$ is a generalized McDuff ultraproduct corresponding to $\balpha^\#\in 2^k$ and $\Gamma$ and that $\balpha^\#(k-1)=1$.

\noindent \textbf{Claim 2:}  $\m_1\equiv_{5k-2}^{EF} \n_1$.  

\noindent \textbf{Proof of Claim 2:}  We view any round $p$ in $\G(\m_1,\n_1,5k-2)$ as round $p+4$ in $\G(\m,\n,5k+3)$ where the first four rounds are played out as above II plays according to the winning strategy for that game.  A priori II's moves come from $\m$ or $\n$ but if they do not land in $\m_1$ or $\n_1$ then I can win the game in 1 more step since $p+4+1\leq 5k-2+5 = 5k+3$ and this would be a contradiction.

Since $5k-2=5(k-1)+3$, by induction we see that $\n_1$ has $\V$ at depth $k-1$.  Since we have $C(\vec{b_2})' \cap C(\vec{b_1}) \subset \n_1 \subset C(\vec{b_3})' \cap C(\vec{b_0})$, it follows that $\Phi(\vec{b_0},\vec{b_1},\vec{b_2},\vec{b_3};k)$.
\end{proof}

\begin{cor}
Suppose that $\balpha,\bbeta\in 2^\omega$ and $k$ are such that $\balpha|k=\bbeta|k$, $\balpha(k)=1$, $\bbeta(k)=0$.  Then $\m_{\balpha}\not\equiv_{5k+3}\m_{\bbeta}$.
\end{cor}

\begin{proof}
Fix $\u\in \beta\N\setminus \N$.  If $\m_{\balpha}\equiv_{5k+3}\m_{\bbeta}$, then $\m_{\balpha}^\u\equiv_{5k+3}^{EF}\m_{\bbeta}^\u$.  Since $\m_{\balpha}=L(K_{\balpha|(k+1)}(\Gamma))$ for some group $\Gamma$, the previous theorem implies that $\m_{\bbeta}^\u$ has $\V$ at depth $k$, contradicting Fact \ref{mainfact}.
\end{proof}

\section{Miscellanea}

\subsection{Distinguishing $\tilde{V}$ with a sentence}

As mentioned in the introduction, it would be interesting to find concrete sentences that are actually distinguishing the McDuff examples.  In this subsection, we show how we can find a set of sentences to distinguish $\m_{\balpha}$ from $\m_{\bbeta}$ when $\balpha(0)=1$ and $\bbeta(0)=0$.

Suppose that $\m$ is a separable McDuff II$_1$ factor for which $\m^\u$ has $\V$ as witnessed by separable $A\subseteq \m^\u$.  Since any separable subalgebra of $\m^\u$ containing $A$ also witnesses that $\m^\u$ has $\V$, by considering a separable elementary substructure of $\m^\u$ containing $A$, we may assume that $A$ is a separable McDuff II$_1$ factor, whence singly generated, say by $a\in A$.  Fix $n\in \N$ and let $\theta_n(w)$ be the meta-statement 
$$\forall \vec x,\vec y\left(\max_{1\leq i \leq n}\|[w,x_i]\|_2=0\rightarrow \inf_{u\in U}\max_{1\leq i,j \leq n}\|[ux_iu^*,y_j]\|_2=0\right).$$  Now $\theta_n(w)$ is not an official statement of continuous logic, but \cite[Proposition 7.14]{bbhu} together with the fact that $(\m^\u,a)$ is $\aleph_1$-saturated and $\theta_n(a)$ holds in $(\m^\u,a)$ implies that there are continuous functions $\gamma_n:[0,1]\to [0,1]$ with $\gamma_n(0)=0$ such that $\psi_n(a)^{(\m^\u,a)}=0$ for each $n$, where $\psi_n(w)$ is the formula
$$\sup_{\vec x,\vec y}\left(\left(\inf_u\max_{1\leq i,j\leq n}\|ux_iu^*,y_j]\|_2\right)\dotminus \gamma_n\left(\max_{1\leq i\leq n}\|[w,x_i]\|_2\right)\right).$$

\begin{prop}
Suppose that $\balpha,\bbeta\in 2^\omega$ are such that $\balpha(0)=1$ and $\bbeta(0)=0$.  Then there are $\gamma_n$, $\psi_n$ as above such that:
\begin{enumerate}
\item For each $n\geq 1$, $(\inf_w \psi_n(w))^{\m_{\balpha}}=0$.
\item There is $n\geq 1$ such that $(\inf_w \psi_n(w))^{\m_{\bbeta}}\not=0$.
\end{enumerate}
\end{prop}

\begin{proof}
Let $\m:=\m_{\balpha}^\u$, $\n:=\m_{\bbeta}^\u$.  Then $\m$ has $\V$, whence the discussion preceding the current proposition holds and we have $\gamma_n$, $\psi_n$ satisfying (1).  Suppose, towards a contradiction, that (2) fails, namely that $(\inf_w \psi_n(w))^{\m_{\bbeta}}=0$ for all $n$.  We claim that $\n$ has $\V$, a contradiction.  By saturation (together with the fact that the $\psi_n$'s get successively stronger), there is $a_1\in \n$ such that $\psi_n(a_1)=0$ for all $n$.  Let $A_1$ be the subalgebra of $\n$ generated by $a_1$.  We claim that $A_1$ witnesses that $\n$ has $\V$.  Towards this end, fix separable $B\subseteq A_1'\cap \n$ and separable $C\subseteq \n$.  Let $\vec{b}$ and $\vec c$ enumerate countable dense subsets of $B$ and $C$ respectively.  Set
$$\Omega(u):=\{u\in U\}\cup\{\|[ub_iu^*,c_j]\|_2=0 \ : \ i,j\in \n\}.$$
By choice of $a_1$, $\Omega(u)$ is finitely satisfiable in $\n$, whence satisfiable in $\n$; if $u$ satisfies $\Omega$, then $uBu^*\subseteq C'\cap \n$, yielding the desired contradiction.
\end{proof}   

Notice that each $\inf_w\psi_n(w)$ has depth $3$ which agrees with the 3 appearing in Proposition \ref{V3}.  Also note that the above discussion goes through with $\m_{\balpha}^\u$ replaced with any generalized McDuff ultraproduct corresponding to $\balpha$ and any countable group $\Gamma$ and likewise for $\m_{\bbeta}^\u$.

\subsection{Good unitaries and definable sets}
We would like to draw the reader's attention to some of the underlying model theory in \cite{BCI} and recast Theorem \ref{main}.  We highlight and give a name to the following concept that played a critical role in \cite{BCI}.
\begin{df}
We say that a pair of unitaries $u,v $ in a II$_1$ factor $\m$ are {\em good unitaries} if $C(u,v)$ is a (2,100)-residual subalgebra of $\m$ (in the terminology of \cite{BCI}) with respect to the unitaries $u$ and $v$, that is, for all $\zeta \in \m$,
\[ 
\inf_{\eta\in C(u,v)}\|\zeta -\eta\|_2 \leq 100 (\|[\zeta,u]\|_2^2 + \|[\zeta,v]\|_2^2). 
\]
 We will call $C(u,v)$ a {\em good subalgebra} with respect to $u$ and $v$.
\end{df}

If $u$ and $v$ are good unitaries, then $C(u,v)$ is a $\{u,v\}$-definable set, which follows immediately from \cite[Proposition 9.19]{bbhu}.  Moreover, we claim that if $u_1,v_1$ are another pair of good unitaries for which $C(u_1,v_1)\subseteq C(u,v)$, then $C(u_1,v_1)'\cap C(u,v)$ is $\{u,v,u_1,v_1\}$-definable.  To see this, we first recall the following fact, due to Sorin Popa and communicated to us by David Sherman.

\begin{fact}
Suppose that $\m$ is a tracial von Neumann algebra with subalgebra $\n$.  Let $E:\m\to \n'\cap \m$ denote the conditional expectation map.  Then for any $x\in \m$, we have
$$\|E(x)-x\|_2\leq \sup_{y\in \n_{\leq 1}}\|[x,y]\|_2.$$
\end{fact}

Note already that this fact shows $C(u,v)'\cap \m$ is $\{u,v\}$-definable for any pair of good unitaries $u,v$.  In general, intersections of definable subsets of metric structures need not be definable, so to show that $C(u_1,v_1)'\cap C(u,v)$ is definable, we need to do a bit more.

\begin{lemma}
Suppose that $\m$ is a tracial von Neumann algebra with subalgebra $\n$.  Let $E:\m\to \n$ denote the conditional expectation and let $P(x):=d(x,\n)$ for all $x\in \m$.  Then $E$ is an $A$-definable function if and only if $P$ is an $A$-definable predicate.
\end{lemma}

\begin{proof}
If $E$ is an $A$-definable function, then $P(x):=\|E(x)-x\|_2$ is an $A$-definable predicate.  Conversely, if $P$ is an $A$-definable predicate, then for any $x,y\in \m$, we have $\|E(x)-y\|_2^2=\|x-y\|_2^2-P(x-y)^2+P(y)^2$, whence $E$ is an $A$-definable function.
\end{proof}

\begin{lemma}
Suppose that $\m$ is a tracial von Neumann algebra with subalgebras $\n_1\subseteq \n_2\subseteq \m$.  Furthermore suppose that $\n_1$ and $\n_2$ are $A$-definable subsets of $\m$.  Then $\n_2'\cap \n_1$ is an $A$-definable subset of $\m$.
\end{lemma}  

\begin{proof}
Since $\n_2'\cap \n_1=(\n_2'\cap \m)\cap \n_1$ and the intersection of two zerosets is again a zeroset, it suffices to show that the distance to $\n_2'\cap \n_1$ is a definable predicate.  To keep things straight, let $E_1:\m\to \n_1$ and $E_2:\m\to \n_2'\cap \n_1$ denote the respective conditional expectations.  By assumption, $E_1$ is $A$-definable.  If $x\in \m$, we have
\begin{alignat}{2}
\|E_2(x)-x\|_2&\leq \|E_2(x)-E_2(E_1(x))\|_2+\|E_2(E_1(x))-E_1(x)\|_2+\|E_1(x)-x\|_2 \notag \\ \notag
		&\leq 2\|x-E_1(x)\|_2+\sqrt{\sup_{y\in (\n_2)_{\leq 1}}\|[E_1(x),y]\|_2}.
\end{alignat}
Since $\n_2$ is an $A$-definable set and $E_1$ is an $A$-definable function, we see that $\n_2'\cap \n_1$ is an $A$-definable set.
\end{proof}

In particular, if $u,v,u_1,v_1$ are as above, then $C(u_1,v_1)'\cap C(u,v)$ is an $\{u,v,u_1,v_1\}$-definable subset of $\m$.


We note that Fact \ref{key.fact} (and the proof of Lemma 2.9 of \cite{BCI}) shows that a special subalgebra of a generalized McDuff ultraproduct is a good subalgebra with respect to some pair of good unitaries.  In the definition of $\V$ at depth $k$, one could modify the definition to only work with pairs of good unitaries instead of arbitrary countable tuples.  It follows from the work in \cite{BCI} that if $\m$ is a generalized McDuff ultraproduct, then $\m$ has $\V$ at depth $k$ if and only if $\m$ has $\V$ at depth $k$ in this augmented sense.

Returning now to the proof of Theorem \ref{main}, by the previous paragraph we see that at each play of the game, I could have chosen a pair of good unitaries instead of a countable sequence.  Moreover, player I would also choose good unitaries corresponding to special subalgebras whenever they played a special subalgebra.  Since II has a winning strategy by assumption, it follows that II always responds with pairs of good unitaries.  Indeed, suppose that I plays good unitaries u,v (say in $\m$) and then II responds with $u_1,v_1\in \n$.  Since II wins, we have that $u_1,v_1$ are unitaries.  To see that they are good, we need to play two more rounds of a side-game.  Fix $\zeta_1 \in \n$ and $\epsilon>0$.  Have I play $\zeta_1$ and have II reply with $\zeta\in \m$.  Since $u,v$ are good, there is $\eta\in C(u,v)$ such that
$$\|\zeta-\eta\|_2<100(\|[\zeta,u\|_2^2+\|[\zeta,v]\|_2^2)+\epsilon.$$  Have I play $\eta$ and II responds with $\eta_1\in \n$.  It follows that $\eta_1\in C(u_1,v_1)$ and $$\|\zeta_1-\eta_1\|_2<100(\|[\zeta_1,u_1\|_2^2+\|[\zeta_1,v_1]\|_2^2)+\epsilon.$$  Since $\zeta_1\in \n$ and $\epsilon>0$ were arbitrary, it follows that $u_1,v_1$ are good.

We see then that the subalgebras called $\m_1$ and $\n_1$ in the proof were in fact definable subalgebras defined over the parameters picked during the game.

We conclude this subsection with a discussion of what goes wrong when trying to distinguish $\m_{\balpha}$ and $\m_{\bbeta}$ with a sentence when $\balpha(0)=\bbeta(0)$ but $\balpha(1)=1$ and $\bbeta(1)=0$.  Motivated by the game played in the previous section, it seems natural to try to use sentences of the form
$$\sup_{u_1,v_1}\inf_{u_2,v_2}\sup_{u_3,v_3}\inf_{u_4,v_4} \chi,$$ where at every stage we quantify only over good unitaries above the previous unitaries in the partial order on tuples and $\chi$ expresses that $C(u_4,v_4)'\cap C(u_2,v_2)$ has $\V$.  There is no issue in saying that the unitaries involved are good and get progressively stronger; moreover, if the unitaries ``played'' at the $\inf$ stages yield a special subalgebra, then $C(u_4,v_4)'\cap C(u_2,v_2)$ is definable and so one can relativize the sentences from the previous subsection to this definable set and indeed express that this commutant has $\V$.  The issue arises in that there were ``mystery'' connectives $\gamma_n$ used in the sentences from the previous subsection and for different choices of good unitaries $u_3,v_3$, the generalized McDuff ultraproducts corresponding to $\balpha^\#$, $C(u_4,v_4)'\cap C(u_2,v_2)$, may require different connectives to express that they have $\V$.  Of course, a positive answer to the following question alleviates this concern and shows how one can find sentences distinguishing $\m_{\balpha}$ from $\m_{\bbeta}$ when $\balpha$ and $\bbeta$ differ for the first time at the second digit (and by induction one could in theory find sentences distinguishing all McDuff examples):     

\begin{question}
Given $\balpha\in 2^\omega$ and a countable group $\Gamma$, are all generalized McDuff ultraproducts corresponding to $\balpha$ and $\Gamma$ elementarily equivalent?
\end{question}

\subsection{Inner asymptotic commutativity and super McDuffness}

Motivated by Sakai's definition of asymptotically commutative II$_1$ factors from \cite{S}, Zeller-Meier introduced the following notion in \cite{ZM}:

\begin{df}
Suppose that $\m$ is a separable II$_1$ factor.  We say that $\m$ is \emph{inner asymptotically commutative} (IAC) if and only if there is a sequence of unitaries $(u_n)$ such that, for all $x,y\in \m$, we have $\lim_n \|[u_nxu_n^*,y]\|_2=0.$
\end{df}

\begin{prop}
Inner asymptotic commutativity is an axiomatizable property.
\end{prop}

\begin{proof}
For $n\geq 1$, consider the sentence 
$$\sigma_n:=\sup_{\vec x,\vec y}\inf_u \max_{1\leq i,j\leq n}\|[ux_iu^*,y_j]\|_2.$$  We claim that a separable II$_1$ factor $\m$ is IAC if and only if $\sigma_n^\m=0$ for all $n$.  The forward implication is clear.  For the converse, suppose that $\sigma_n^\m=0$ for all $n$.  Let $\{a_i \ : \ i\in \n\}$ be a dense subset of $\m$.  For each $n$, let $u_n\in U(\m)$ be such that $\|[u_na_iu_n^*,a_j]\|_2<1/n$ for all $i,j\leq n$.  It then follows that $(u_n)$ witnesses that $\m$ is IAC.  
\end{proof}


Zeller-Meier also considers another property that may or may not hold for separable II$_1$ factors.  Before we can define this property, we need some preparation:

\begin{prop}
Suppose that $\m$ is a separable McDuff II$_1$ factor and $\m\preceq \C\preceq \tC$ with $\C$ and $\tC$ both $\aleph_1$-saturated.  Then the following are equivalent:
\begin{enumerate}
\item $Z(\m'\cap \C)=\mathbb{C}$
\item $Z(\m'\cap \tC)=\mathbb{C}$.
\end{enumerate}  
\end{prop}

\begin{proof}
First suppose that (2) fails, so there is $a\in Z(\m'\cap \tC)$ such that $d(a,\tr(a)\cdot 1)\geq \epsilon$.  Since $\m$ is McDuff, it is singly generated, say by $m\in \m$.  Since $(\tC,a,m)$ is $\aleph_1$-saturated, there is a continuous function $\gamma:[0,1]\to [0,1]$ with $\gamma(0)=0$ such that 
$$(\tC,a,m)\models \sup_y (\|[a,y]\|_2\dotminus \gamma(\|[y,m]\|_2)=0.$$  It follows that
$$(\tC,m)\models \inf_x\max(\|[x,m]\|_2,\sup_y(\|[x,y]\|_2\dotminus \gamma(\|[y,m]\|_2),\epsilon\dotminus d(x,\tr(x)\cdot 1)=0.$$  By elementarity, the same statement holds in $(\C,m)$; by saturation, the infimum is realized by $b\in \C$.  It follows that $b\in Z(M'\cap \C)\setminus \mathbb{C}$, so (1) fails.

Now suppose that (2) holds and consider $a\in Z(M'\cap \C)$.  Then there is a continuous function $\eta:[0,1]\to [0,1]$ with $\eta(0)=0$ such that
$$\C\models \sup_y(\|[y,a]\|_2\dotminus \eta(\|[y,m]\|_2))=0.$$  By elementarity, the same statement holds in $\tC$, that is, $a\in Z(M'\cap \tC)=\mathbb{C}$, whence (1) holds.
\end{proof}

Observe that the end of the above proof actually shows that, under the same hypotheses as in the proposition, we have $Z(M'\cap \C)\subseteq Z(M'\cap \tC)$.
\begin{cor}
Suppose that $\m$ is a separable McDuff II$_1$ factor.  Then the following are equivalent:
\begin{enumerate}
\item $Z(\m'\cap \C)=\mathbb{C}$ for every $\aleph_1$-saturated elementary extension $\C$ of $\m$.
\item $Z(\m'\cap \C)=\mathbb{C}$ for some $\aleph_1$-saturated elementary extension $\C$ of $\m$.
\end{enumerate}
\end{cor}

\begin{df}
If $\m$ is a separable McDuff II$_1$ factor, we say that $\m$ is \emph{super McDuff} if either of the equivalent conditions of the previous corollary hold.
\end{df}

It would be nice to know if being super McDuff is axiomatizable, for then \cite{ZM} gives another example of a theory of II$_1$ factors. At the moment, the following proposition is the best that we can do.

\begin{prop}\label{super}
Suppose that $\m$, $\n$ are separable McDuff II$_1$ factors with $\m\preceq \n$.  If $\n$ is super McDuff, then so is $\m$.
\end{prop}

First, we need a little bit of preparation.  Given $p\in S(\m)$, we define $p^\u\in S(\m^\u)$ by declaring, for every formula $\varphi(x,y)$ and every element $a:=(a_i)^\bullet\in \m^\u$, $\varphi(x,a)^{p^\u}:=\lim_\u \varphi(x,a_i)^p$.

\begin{lemma}\label{notalgebraic}
If $p\in S(\m)$ is not algebraic, then neither is $p^\u\in S(\m^\u)$.
\end{lemma}

\begin{proof}
Suppose that $p^\u$ is algebraic.  Let $\n$ be an elementary extension of $\m$ containing a realization $a$ of $p$.  Then $a^\bullet\in \n^\u$ is a realization of $p^\u$, whence it belongs to $\m^\u$ by algebraicity of $p^\u$.  It follows that $a$ is the limit of a sequence from $\m$, whence it belongs to $\m$ as well.  Since $a$ was an arbitrary realization of $p$, we conclude that $p$ is algebraic.
\end{proof}

\begin{proof}[Proof of Proposition \ref{super}]
Fix a nonprincipal ultrafilter $\u$ on $\N$.  Without loss of generality, we may assume that $\m\preceq \n\preceq \m^\u$.  Suppose that $\m$ is not super McDuff as witnessed by $a\in Z(\m'\cap \m^\u)\setminus \mathbb{C}$.  Let $p:=\tp(a/\m)$; since $\m$ is a II$_1$ factor, $p$ is not algebraic, whence neither is $p^\u$.  Let $\C$ be a $(2^{\aleph_0})^+$-saturated elementary extension of $\m^\u$.

\noindent \textbf{Claim 1:}  $p^\u(\C)\subseteq (\m^\u)'\cap \C$.

\noindent \textbf{Proof of Claim 1:}  Let $\varphi(x,y)$ denote the formula $\|[x,y]\|_2$.  Then for any $b\in \m$, we have that $\varphi(x,b)^p=0$, whence it follows that for any element $b\in \m^\u$ we have $\varphi(x,b)^{p^\u}=0$, verifying the claim.

%
%

\noindent \textbf{Claim 2:}  $p(\m^\u)\subseteq Z(\m'\cap \m^\u)$.

\noindent \textbf{Proof of Claim 2:}  Fix $\epsilon>0$.  Then the following set of conditions is unsatisfiable in $\m^\u$:
$$\{\|[x,b]\|_2=0 \ : \ b\in \m\}\cup\{\|[x,a]\|_2\geq \epsilon\}.$$  By saturation, there are $b_1,\ldots,b_n\in \m$ such that the following meta-statement is true in $\m^\u$:
$$\m^\u\models \forall x\left(\max_{1\leq i\leq n}\|[x,b_i]\|_2=0\rightarrow (\|[x,a]\|_2\dotminus \epsilon)=0\right).$$  As above, by saturation this meta-statement can be made into an actual first-order formula with parameters from $\m$ that holds of $a$, whence it holds of any other realization of $p$ in $\m^\u$.  This shows that if $a'\in p(\m^\u)$ and $c\in \m'\cap \m^\u$, then $\|[a',c]\|_2\leq \epsilon$; since $\epsilon>0$ is arbitrary, this proves the claim.

\noindent \textbf{Claim 3:}  $p^\u(\C)\subseteq Z((\m^\u)'\cap \C)$.

\noindent \textbf{Proof of Claim 3:}  Suppose that $a'\in \C$ realizes $p^\u$.  Fix $b'\in (\m^\u)'\cap \C$ .  Take $a'',b''\in \m^\u$ such that $\tp(a',b'/\m)=\tp(a'',b''/\m)$.  By Claim 2, $a''\in Z(\m'\cap \m^\u)$.  Note also that $b''\in \m'\cap \m^\u$.  It follows that $\|[a',b']\|_2=\|[a'',b'']\|_2=0$, yielding the desired conclusion.



In order to establish that $\n$ is not super McDuff, by Lemma \ref{notalgebraic}, it suffices to establish the following claim:

\noindent \textbf{Claim 4:} $p^\u|\n(\m^\u)\subseteq Z(\n'\cap \m^\u)$. 

\noindent \textbf{Proof of Claim 4:}
Arguing as in the proof of Claim 1, we see that $p^\u|\n(\m^\u)\subseteq \n'\cap \m^\u$.  Now suppose that $a'\in p^\u|\n(\m^\u)$ and $b'\in \n'\cap \m^\u$.  Then $a'\in Z(\m'\cap \m^\u)$ by Claim 2 and $b'\in \m'\cap \m^\u$, so $[a,b]=0$ as desired.  
\end{proof}

\subsection{The first-order fundamental group} For a II$_1$ factor $\m$ and $t\in \R_+$, we let $\m_t$ denote the amplification of $\m$ by $t$.  Note that if $\u$ is an ultrafilter, then $(\m^\u)_t$ is canonically isomorphic to $(\m_t)^\u$, whence we can unambiguously write $\m_t^\u$.  

Recall that the fundamental group of $\m$ is the set $\F(\m):=\{t\in \R_+ \ : \ \m_t\cong \m\}$.  $\F(\m)$ is a (not necessarily closed) subgroup of $\R_+$.  We now consider the \emph{first-order fundamental group of $\m$}, $\FF(\m):=\{t\in \R_+ \ : \ \m_t\equiv \m\}$.  Clearly $\F(\m)\subseteq \FF(\m)$.  As the name indicates, $\FF(\m)$ is actually a group.  The easiest way to see this is to recognize that $\FF(\m)$ is absolute, whence, assuming CH, we have $\FF(\m)=\F(\m^\u)$ for a fixed ultrafilter $\u$ on $\N$.  Alternatively, one can use Keisler-Shelah as follows.  Suppose that $s,t\in \FF(\m)$.  By Keisler-Shelah, there is $\u$ such that $\m^\u\cong \m^\u_s$.  Note now that $\m^\u\equiv \m^\u_t$, whence there is $\v$ such that $\m^\u\cong (\m^\u_t)^\v$.  We then have
$$(\m^\u)^\v\cong (\m^\u_t)^\v\cong ((\m^\u_s)_t)^\v=(\m^\u_{st})^\v\cong ((\m_{st})^\u)^\v,$$ whence it follows that $\m\equiv \m_{st}$.

Unlike the ordinary fundamental group, the first-order fundamental group is a closed subgroup of $\R_+$.  Indeed, if $(r_k)$ is a sequence from $\R_+$ with limit $r\in \R_+$, it is easy to verify that $\prod_\u \m_{r_k}\cong \m_r^\u$ for any nonprincipal ultrafilter $\u$ on $\N$; if each $r_k\in \FF(\m)$, then $\prod_\u \m_{r_k}\equiv \m$, whence it follows that $r\in \FF(\m)$.

In summary:

\begin{prop}
$\FF(\m)$ is a closed subgroup of $\R_+$ containing $\F(\m)$.
\end{prop}

\begin{question}
Does there exist a separable II$_1$ factor $\m$ for which $\FF(\m)\not=\R_+$?
\end{question}

Recall that II$_1$ factors $\m$ and $\n$ are said to be \emph{stably isomorphic} if $\m\cong \n_t$ for some $t\in \R_+$.  So the above question is equivalent to the question:  does stable isomorphism imply elementary equivalence?  Since all of the free group factors are stably isomorphic, a special case of the above question is whether or not all of the free group factors are elementarily equivalent (a question Thomas Sinclair has called the \emph{noncommutative Tarski problem}).

In connection with the number of theories of II$_1$ factors, we have:

\begin{prop}
Suppose that $\m$ is a separable II$_1$ factor with $\FF(\m)\not=\R_+$.  Then $$|\{\Th(\m_t) \ : \ t\in \R_+\}|=2^{\aleph_0}.$$
\end{prop}

\begin{proof}
Since the map $t\FF(\m)\mapsto \Th(\m_t)$ is injective, the result follows from the fact that closed subgroups of $\R_+$ are countable.
\end{proof}

It seems very unlikely that $\FF(\m)=\R_+$ for all separable II$_1$ factors $\m$.  In fact, it seems very unlikely that $\m\equiv \m_2(\m)$ for all separable II$_1$ factors $\m$.  Let $\FA(\m):=\{t\in \R_+ \ : \ \m\equiv_\forall \m_t\}$.  Of course, if CEP holds, then $\FA(\m)=\R_+$ for any II$_1$ factor $\m$, so what follows is only interesting if CEP fails.

\newpage
\begin{prop}\label{equivalence}
The following statements are equivalent:
\begin{enumerate}
\item If $\m$ is existentially closed, then $\m$ is McDuff.
\item If $\m$ is existentially closed, then $2\in \FA(\m)$.
\item For any II$_1$ factor $\m$, $2\in \FA(\m)$.
\item For any II$_1$ factor $\m$, $\m\equiv_\forall \m\otimes \r$.
\item For any II$_1$ factor $\m$, $\FA(\m)=\R_+$.
\end{enumerate}
\end{prop}

In the statement of the proposition, when we say that $\m$ is existentially closed, we mean that $\m$ is an existentially closed model of its theory.

\begin{proof}[Proof of Proposition \ref{equivalence}]
Since McDuff II$_1$ factors have full fundamental group, (1) implies (2) is trivial.  (2) implies (3) follows from the fact that $\m\equiv_\forall \n$ implies $\FA(\m)=\FA(\n)$.  (3) implies (4) follows from the fact that $\m\otimes \r$ embeds into $\prod_\u M_{2^n}(\m)$.  Now suppose that (4) holds and fix an arbitrary II$_1$ factor $\m$.  Since $\m\otimes \r$ is McDuff, for any $t\in \R_+$ we have that $$\m\equiv_\forall \m\otimes \r \cong (\m\otimes \r)_t\equiv_\forall \m_t,$$ whence (5) holds.  Finally assume that (5) holds and assume that $\m$ is existentially closed.  By considering the chain 
$$\m\subseteq M_2(\m)\subseteq M_4(\m)\subseteq M_8(\m)\subseteq \cdots$$ and noting that each element of the chain has the same universal theory as $\m$ by (5), we see that $\m$ is existentially closed in the union $\m\otimes \r$.  Since $\m\otimes \r$ is McDuff and being McDuff is $\forall\exists$-axiomatizable, we have that $\m$ is McDuff as well.
\end{proof}

Note that it is not always true that $\m\equiv_{\forall\exists}\m\otimes \r$ (e.g. when $\m$ is not McDuff).

\end{document}